\documentclass[reqno]{amsart}
\usepackage[english]{babel}
\usepackage{amssymb}

\newtheorem{theorem}{Theorem}[section]
\newtheorem{proposition}[theorem]{Proposition}
\newtheorem{lemma}[theorem]{Lemma}
\newtheorem{corollary}[theorem]{Corollary}

\theoremstyle{definition}
\newtheorem{definition}[theorem]{Definition}
\newtheorem{example}[theorem]{Example}

\theoremstyle{remark}
\newtheorem*{remark}{Remark}

\numberwithin{equation}{section}

\newcommand{\RE}{\mbox{$\mathbb{R}$}}
\newcommand{\CEP}[1]{\mbox{$\mathbb{C}^{#1}$}}

\newcommand{\N}{\mbox{$\mathbb{N}$}}
\newcommand{\D}{\mbox{$\mathbb{D}$}}

\newcommand{\supp}{\mbox{supp}}

\newcommand{\comp}[1]{\mbox{$\mathcal{C}{#1}$}}

\newcommand{\Ker}{\mbox{$\mathcal{N}$}}
\newcommand{\E}{\mbox{$\mathcal{E}$}}
\newcommand{\Eo}{\mbox{$\mathcal{E}_{0}$}}
\newcommand{\F}{\mbox{$\mathcal{F}$}}

\newcommand{\MPSH}[1]{\mbox{$\mathcal{MPSH}(#1)$}}
\newcommand{\PSH}[1]{\mbox{$\mathcal{PSH}(#1)$}}

\newcommand{\K}{\mbox{$\mathcal{K}$}}

\newcommand{\B}{\mbox{$\mathcal{B}$}}

\newcommand{\ddcn}[1]{\mbox{$\left(dd^{c}#1\right)^{n}$}}
\newcommand{\ddck}[2]{\mbox{$\left(dd^{c}#1\right)^{#2}$}}

\begin{document}

\title{Monge-Amp\`{e}re measures on pluripolar sets}
\author{Per \AA hag}
\address{Department of Mathematics\\ Ume\aa \ University\\
SE-901 87 Ume\aa \\ Sweden}
\author{Urban Cegrell}
\address{Department of Mathematics\\ Ume\aa \ University\\ SE-901 87 Ume\aa \\ Sweden}
\email{Urban.Cegrell@math.umu.se}
\author{Rafa\l\ Czy{\.z}}
\address{Institute of Mathematics\\ Jagiellonian University\\ Reymonta 4\\ 30-059 Krak\'ow\\ Poland}
\email{Rafal.Czyz@im.uj.edu.pl}
\thanks{The third-named author was partially supported by ministerial grant number N N201 3679 33}
\author{Pham Hoang Hiep}
\address{Department of Mathematics\\
Truong Dai hoc Su pham Ha Noi\\ 136 Xuan Thuy, Cau Giay, Ha Noi \\
Vietnam} \email{phhiep\_vn@yahoo.com}

\begin{abstract} In this article we solve
the complex Monge-Amp\`{e}re equation for measures with large
singular part. This result generalizes classical results by
Demailly, Lelong and Lempert a.o., who considered singular parts
carried on discrete sets. By using our result we obtain a
generalization of Ko{\l}odziej's subsolution theorem. More
precisely, we prove that if a non-negative Borel measure is
dominated by a complex Monge-Amp\`{e}re measure, then it is a
complex Monge-Amp\`{e}re measure.

\end{abstract}

\maketitle

\section{Introduction}

In this article we study the complex Monge-Amp\`{e}re equation
$\ddcn{u}=\mu$, where $\mu$ is a given non-negative Radon measure
and $\ddcn{\,\cdot\,}$ denotes the complex Monge-Amp\`{e}re
operator. If $\mu$ puts mass on a pluripolar set, then the
solution to $\ddcn{u}=\mu$ cannot generally be uniquely
determined (see, e.g.~\cite{demailly_bok,zeriahi_mich}). Therefore
the question of existence of solutions is our main interest. The
first result was due to Lempert who, in~\cite{lempert1,lempert2}, obtained
 a positive result for the case when
the support of the given measure is a single point. He considered
solutions with real-analytic boundary values and logarithmic
singularity near the support of the measure. The underlying
domain was assumed to be a strictly convex domain in $\CEP{n}$
(see also~\cite{bracci} and Theorem~1.5 in~\cite{klimek_green}).
In this context, it is worth to mention the article~\cite{celik_poletsky}, where Celik and
Poletsky also studies the Monge-Amp\`{e}re equation with the Dirac measure as given measure.

Throughout this article it is always assumed that $\Omega$ is a bounded
hyperconvex domain (see section~\ref{sec_maj} for the definition of hyperconvex domain).
Demailly proved (Theorem~4.3 in~\cite{demailly_Z}) that $\ddcn{g_{A_1}}=(2\pi)^n
\delta_z$ on a hyperconvex domain $\Omega$, where $\delta_z$ is
the Dirac measure at $z$, and $g_z$ is the pluricomplex Green
function (introduced in~\cite{klimek,zahariuta}) with pole set
containing a single point $A_1=\{z\}$. In~\cite{lelong2}, Lelong
introduced the pluricomplex Green function with a finite pole set,
$A_k=\{z_1,\ldots,z_k\}$, and with positive weights
$v_1,\ldots,v_k$, $v_l>0$, $l=1,\ldots, k$, and proved that
$\ddcn{g_{A_k}}=(2\pi)^n\sum_{j=1}^k v_j^n\delta_{z_j}$
(Proposition~8 in~\cite{lelong2}). The pluricomplex Green function
is not a solution  to the complex Monge-Amp\`{e}re equation if we
want the solution to have other boundary values than those which are
identically zero. Given a discrete measure with
compact support in a hyperconvex domain
$\Omega$, Zeriahi proved in~\cite{zeriahi_mich} that the
complex Monge-Amp\`{e}re equation is solvable for certain continuous
boundary values. In~\cite{xing_sing}, Xing generalized Zeriahi's
result in the case where the given boundary values are identically
zero. Xing considered measures that were majorized by the sum of a
linear combination of countable numbers of Dirac measures with
compact support and a certain regular Monge-Amp\`{e}re
measure.

 We shall consider the class $\E$ introduced in~\cite{cegrell_gdm}. It is
the largest set of non-positive
plurisubharmonic functions defined on a hyperconvex domain
$\Omega$ for which the complex Monge-Amp\`{e}re operator is
well-defined  (Theorem~4.5 in~\cite{cegrell_gdm}). Let $u\in\E$
and $0\leq g\leq 1$ be a $\chi_{\{u=-\infty\}}\ddcn{u}$-measurable
function that vanishes outside $\{u=-\infty\}$. We define
\[
u^g=\inf_{f\in T \atop f\leq g} \left(\sup\{u_\tau: f\leq \tau,\;
\tau \text{ is a bounded lower semicontinuous
function}\}\right)^*\, ,
\]
where $u_\tau$ is as in Definition~\ref{DP_def_aux1}, $T$ is the
family of certain simple functions, and $(w)^*$ denotes the upper
semicontinuous regularization of $w$. We prove that $u^g\in\E$ and
$\ddcn{u^g}=g\ddcn{u}$. In particular, this implies that for any
pluripolar Borel set $E$ in $\Omega$ we have that
$\ddcn{u^{\chi_E}}=\chi_E\ddcn{u}$, where $\chi_E$ is the
characteristic function for the set $E$ in $\Omega$
(Theorem~\ref{DP_thm_simple}). Example~\ref{DP_exem} shows that
our given singular measure $\chi_E\ddcn{u}$ is not necessarily a
discrete measure. Hence, Theorem~\ref{DP_thm_simple} yields
solutions to the complex Monge-Amp\`{e}re equation for a larger
class of singular measures than Zeriahi and Xing. The following statement,
which is included in Theorem~\ref{DP_thm_sub}, is the main result of this
article:

\bigskip

\noindent {\bf Theorem~\ref{DP_thm_sub}} (3): {\em Assume that $\mu$ is a non-negative
Radon measure.  If there exists a function $w\in\E$ such that $\mu\leq\ddcn{w}$, then
there exists a function $u\in\E$ such that $w+H\leq u\leq H$ and $\ddcn{u}=\mu$.}

\bigskip

Theorem~\ref{DP_thm_sub} is a generalization of the celebrated subsolution
theorem by Ko{\l}odziej (\cite{kolo_range};
for an alternative proof see section~4 in~\cite{kolo_mem}).
Example~5.4 in~\cite{cegrell_bdd} shows that there exists a
non-negative Radon measure $\mu$ such that there does not exist any
function $u\in\E$ that satisfies $\ddcn{u}=\mu$.

\bigskip

This article is organized as follows. In Section~\ref{sec_maj}
some definitions will be recalled. One of the most powerful tools
when working with the complex Monge-Amp\`{e}re operator is the
{\em comparison principle}. In Section~\ref{sec_xing} we obtain
the comparison principle for certain functions in $\E$
(Corollary~\ref{xing_cor_CP}). To prove the comparison principle
we shall follow an idea from~\cite{xing_cont} and firstly prove a
Xing type inequality (Theorem~\ref{xing_cor_1}). The last section is devoted
to the proof of Theorem~\ref{DP_thm_sub}.

\bigskip

%
%

The authors would like to thank
Egmont Porten and Alexander Rashkovskii for many valuable comments on and suggestions for this
manuscript. This research was partly done during the third- and
fourth-named authors' visit to Mid Sweden University in Sundsvall,
Sweden, in 2006 and 2007. This article was completed during their
visit to Ume\aa \ University in Ume\aa, Sweden, in 2008. They wish
to thank the members of both Departments of Mathematics for their
kind hospitality.

\section{Background and definitions}\label{sec_maj}
Throughout this article it is always assumed that $\Omega$ is a bounded
hyperconvex domain. Recall that $\Omega\subseteq\CEP{n}$, $n\geq 1$ is a bounded hyperconvex domain
if it is a bounded, connected, and open set, such that there exists a bounded
plurisubharmonic function $\varphi:\Omega\rightarrow (-\infty,0)$ such that
the closure of the set $\{z\in\Omega : \varphi(z)<c\}$ is compact in $\Omega$,
for every $c\in (-\infty, 0)$.

In this article we adapt the notation that  $\PSH{\Omega}$ is the
family of plurisubharmonic functions defined on $\Omega$ and
$\MPSH{\Omega}$ for the maximal plurisubharmonic functions. For the
definitions and basic facts of these functions we refer to~\cite{klimek}.

We say that a
 bounded plurisubharmonic function $\varphi$ defined on $\Omega$ belongs
 to $\Eo$ if $\lim_{z\rightarrow\xi} \varphi (z)=0$, for every
$\xi\in\partial\Omega$, and $\int_{\Omega} \ddcn{\varphi}<
+\infty$. It was proved in Lemma~3.1 in~\cite{cegrell_gdm} that
$C^{\infty}_{0}(\Omega)\subset\Eo\cap C(\bar\Omega)-\Eo\cap
C(\bar\Omega)$.

\begin{definition} Let $\E$ $(=\E (\Omega))$ be the class of plurisubharmonic
functions $\varphi$ defined on $\Omega$, such that for each
$z_{0}\in\Omega$ there exists a neighborhood $\omega$ of $z_{0}$
in $\Omega$ and a decreasing sequence
$[\varphi_{j}]_{j=1}^{\infty}$, $\varphi_{j}\in\Eo$, that
converges pointwise to $\varphi$ on $\omega$ as $j\to +\infty$,
and
\[\sup_{j}\int_{\Omega} \ddcn{\varphi_{j}}<+\infty\, .\]
Furthermore, let $\F$  $(=\F(\Omega))$ be the subset of $\E$  containing those
functions with smallest maximal plurisubharmonic majorant identically zero and with finite total
Monge-Amp\`{e}re mass.
\end{definition}

If there can be no misinterpretation a sequence
$[\,\cdot\,]_{j=1}^{\infty}$ will be denoted by $[\,\cdot\,]$.
Shiffman and Taylor gave an example in~\cite{siu}
 that shows that it is not possible to extend the complex
 Monge-Amp\`{e}re operator in a meaningful way to the whole class of
 plurisubharmonic functions and still have the
range contained in the class of non-negative measures (see
also~\cite{kiselman_ex}). In~\cite{cegrell_gdm} the second-named
author proved that the complex Monge-Amp\`{e}re operator is
well-defined on $\E$. As mentioned in the introduction he
proved that $\E$ is the natural domain of
definition for the complex Monge-Amp\`{e}re operator (Theorem~4.5
in~\cite{cegrell_gdm}). In~\cite{block_C2}, B\l ocki proved that $\E = \{\varphi\in\PSH{\Omega}\cap W^{1,2}_{loc}(\Omega):\; \varphi\leq 0 \}$ when $n=2$,
and showed that this equality
is not valid  for $n\geq 3$. Later, in~\cite{block_Cn}, he obtained
a complete characterization of $\E$ for $n\geq 1$. Another
characterization of $\E$ was proved in~\cite{cegrell_Kolo_sub} in
terms of the so-called $\varphi$-capacity.

In this article a {\em fundamental sequence} $[\Omega_{j}]$
is always an increasing sequence of strictly pseudoconvex
subsets of $\Omega$ such that for every $j\in\N$ we have that,
$\Omega_{j}\Subset\Omega_{j+1}$, and
$\bigcup_{j=1}^{\infty}\Omega_{j}=\Omega$. Here  $\Subset$ denotes
that $\Omega_{j}$ is relatively compact in $\Omega_{j+1}$.

%
%
\begin{definition}\label{maj_def_seq}
Let $u\in\PSH{\Omega}$, $u\leq 0$, and let $[\Omega_j]$ be a
fundamental sequence $\Omega_{j}$ . The function $u^j$ is then
defined by
\[
u^{j} = \sup\big\{\varphi\in\PSH{\Omega}: \varphi\leq
u\;\;\mbox{on}\;\; \comp{\Omega_{j}}\big\}\, ,
\]
where $\comp{\Omega_{j}}$ denotes the complement of $\Omega_{j}$
in $\Omega$.
\end{definition}
Let $[\Omega_j]$ be a fundamental sequence and let
$u\in\PSH{\Omega}$, $u\leq 0$, then $u^j\in\PSH{\Omega}$ and
$u^j=u$ on $\comp{\Omega_{j}}$. Definition~\ref{maj_def_seq}
implies that $[u^j]$ is an increasing sequence and therefore
$\lim_{j\to+\infty}u^j$ exists q.e. (quasi-everywhere) on
$\Omega$. Hence, the function $\tilde u$ defined by $\tilde
u=\left(\lim_{j\to+\infty}u^j\right)^*$ is plurisubharmonic on
$\Omega$. Moreover, if $u\in\E$, then
by~\cite{cegrell_gdm} we have that $\tilde u\in\E$, since $u\leq
\tilde u\leq 0$, and by~\cite{block_C2,block_Cn} it follows that
$\tilde u$ is maximal on $\Omega$. Let $u,v\in\E$ and
$\alpha\in\RE$, $\alpha\geq 0$, then it follows from
Definition~\ref{maj_def_seq} that $\widetilde{u+v} \geq \tilde u +
\tilde v$ and $\widetilde{\alpha\, u}=\alpha\, \tilde{u}$.
Moreover, if $u\geq v$, then $\tilde u\geq \tilde v$. It
follows from~\cite{block_C2,block_Cn} that
$\E\cap\MPSH{\Omega}=\{u\in\E:\tilde u=u\}$ . Set
\[
\Ker=\{u\in\E:\tilde u=0\}\, .
\]
Then we have that $\Ker$ is a convex cone and that $\Ker$ is
precisely the set of functions in $\E$ with smallest maximal
plurisubharmonic majorant identically zero.

\begin{definition}\label{maj_def_N(H)} Let $\K\in \{\, \Eo,\,
\F,\, \Ker\,\}$. We say that a plurisubharmonic function $u$
defined on $\Omega$ belongs to the class $\K(\Omega,H)$, $H\in\E$,
if there exists a function $\varphi\in\K$ such that
\[
H\geq u\geq \varphi + H\, .
\]
\end{definition}

Note that $\K(\Omega,0)=\K$. The following approximation theorem was proved by the
second-named author in~\cite{cegrell_gdm}.

\begin{theorem}\label{maj_thm_appr}
Let $u\in\PSH{\Omega}$, $u\leq 0$. Then there exists a decreasing
sequence $[u_{j}]$, $u_{j}\in\Eo\cap C(\bar\Omega)$, which
converges pointwise to $u$ on $\Omega$, as $j$ tends to $+\infty$.
\end{theorem}

Theorem~\ref{maj_thm_appr} yields among other things the following
simple and useful observation.

\begin{proposition}\label{maj_prop_appr}
Let $H\in\E$ and $u\in \PSH{\Omega}$ be such that $u\leq H$, then
there exists  a decreasing sequence $[u_{j}]$, $u_{j}\in\Eo (H)$,
that converges pointwise to $u$ on $\Omega$, as $j$ tends to
$+\infty$. Moreover, if $H\in\PSH{\Omega}\cap C(\bar\Omega)$, then
the decreasing sequence $[u_{j}]$ can be chosen such  that
$u_j\in\Eo(H)\cap C(\bar\Omega)$.
\end{proposition}
\begin{proof} Theorem~\ref{maj_thm_appr} implies that there
exists a decreasing sequence $[\varphi_{j}]$,
$\varphi_{j}\in\Eo \cap C(\bar\Omega)$, that converges pointwise
to $u$, as $j\to +\infty$. If $v_{j}=\max(u,\varphi_{j}+H)$, then
$[v_{j}]$, $v_{j}\in\Eo (H)$, is a decreasing sequence that
converges pointwise to $u$, as $j\to +\infty$, and the first
statement is completed.

For the second statement assume that $H\in\PSH{\Omega}\cap
C(\bar\Omega)$ and let $\varphi\in\Eo \cap C(\bar\Omega)$, not
identically $0$. Choose a fundamental sequence, $[\Omega_j]$ of
$\Omega$ such that for each $j\in\N$ we have that
$\varphi\geq-\frac{1}{2j^2}$ on $\comp{\Omega_{j}}$. Let
$[v_j]$, $v_j\in \PSH{\Omega_j}\cap C^{\infty}(\Omega_j)$, be a
decreasing sequence that converges pointwise to $u$, as $j\to
+\infty$, and $v_j\leq H+\frac{1}{2j}$ on $\Omega_{j+1}$. Set
\[
u^{\prime}_j=
\begin{cases}
\max\left(v_j-\frac1j\, , j\varphi +H\right) & \text{ on }
\Omega_j\\[2mm]
j\varphi +H & \text{ on } \comp{\Omega_{j}}\, .
\end{cases}
\]
Then $[u^{\prime}_j]$, $u^{\prime}_j\in\Eo(H)\cap C(\bar\Omega)$,
converges pointwise to $u$ on $\Omega$, as $j\to +\infty$, but
$[u^{\prime}_j]$ is not necessarily decreasing. Let
 $u_j=\sup_{k\geq j}u^{\prime}_k$. The construction of
 $u^{\prime}_j$ implies that
 \[
u^{\prime}_j +\frac1j\geq u^{\prime}_{j+1}+\frac{1}{j+1}
 \]
 and therefore for each $j\in\N$ fixed it follows that
\[
 \left[\max\left(u^{\prime}_j,u^{\prime}_{j+1},\ldots,u^{\prime}_{m-1},u^{\prime}_m+\frac1m\right)\right]_{m=j}^\infty
 \]
 decreases pointwise on $\Omega$ to $u_j$, as $m\to +\infty$.
 Thus, $u_j$ is an upper semicontinuous function and we
 have that $u_j\in\PSH{\Omega}\cap C(\bar\Omega)$. Moreover $[u_j]$ is decreasing and
 converges pointwise to $u$ on $\Omega$, as $j\to +\infty$.
\end{proof}
\begin{remark} If $H$ is unbounded, then each function $u_j$ is necessarily
unbounded.
\end{remark}

\section{Some auxiliary results}\label{sec_xing}

\begin{theorem}\label{xing_cor_1} Let $H\in\E$. If $u\in\Ker
(H)$ and $v\in\E$ is such that $v\leq H$ on $\Omega$, then for all
$w_j\in\PSH{\Omega}\cap L^{\infty}(\Omega)$, $-1\leq w_j\leq 0$,
$j=1,2,...,n$, we have the following inequality:
\begin{multline}\label{xing_ineq}
\frac{1}{n!}\int_{\{u<v\}}(v-u)^n dd^c w_1\wedge\cdots\wedge
dd^cw_n +\int_{\{u<v\}}(-w_1)\ddcn{v}
 \leq \\ \leq
 \int_{\{u<v\}}(-w_1)\ddcn{u}+\int_{\{u=v=-\infty\}}(-w_1)\ddcn{u}\,
 .
\end{multline}
\end{theorem}
\begin{proof} Let $u\in\Ker (H)$, i.e., $u\in\PSH{\Omega}$ and
there exists a function $\varphi\in\Ker$ such that
\[
H\geq u\geq \varphi + H\, .
\]
Let $[\Omega_j]$ be a fundamental sequence in $\Omega$ and let
$\varphi^j$ be defined as in Definition~\ref{maj_def_seq}. The assumption that $v\leq H$
implies that for $\varepsilon>0$ the following inequality holds
\[
u\geq \varphi + H =\varphi^j+H\geq
\varphi^j+v-\varepsilon\;\;\text{ on }\comp{\Omega_{j}}\, .
\]
Theorem~4.9 in~\cite{Khue_Hiep} implies that
\begin{multline*}
\frac{1}{n!}\int_{\{u<v-\varepsilon+\varphi^j\}}(v-\varepsilon+\varphi^j-u)^n
dd^c w_1\wedge\cdots\wedge dd^cw_n
+\int_{\{u<v-\varepsilon+\varphi^j\}}(-w_1)\ddcn{v}
 \leq \\ \leq
 \int_{\{u\leq v-\varepsilon\}}(-w_1)\ddcn{u}\,
 .
\end{multline*}
We have that
\begin{equation}\label{xing_cor_3}
[\chi_{\{u<v-\varepsilon+\varphi^j\}}(v-\varepsilon+\varphi^j-u)^n]_{j=1}^{\infty}\;\;\text{
and }\;\;[\chi_{\{u<v-\varepsilon+\varphi^j\}}]_{j=1}^{\infty}
\end{equation}
are two increasing sequences of functions that converges q.e. on
$\Omega$ to $\chi_{\{u<v-\varepsilon\}}(v-\varepsilon-u)^n$ and
$\chi_{\{u<v-\varepsilon\}}$, respectively, as $j\to +\infty$. Theorem~5.11
in~\cite{cegrell_gdm} implies that $dd^c w_1\wedge\cdots\wedge
dd^c w_n\ll C_n$ and $\chi_{\{v>-\infty\}}\ddcn{v}\ll C_n$. Here $C_n$
denotes the usual $C_{n}$-capacity and $\mu\ll C_n$ denotes that the measure
$\mu$ is absolutely continuous with
respect to $C_n$ (see e.g.~\cite{kolo_mem} for background). We therefore have that
$[\chi_{\{u<v-\varepsilon+\varphi^j\}}(v-\varepsilon+\varphi^j-u)^n]_{j=1}^{\infty}$
converges to $\chi_{\{u<v-\varepsilon\}}(v-\varepsilon-u)^n$ a.e.
w.r.t. $dd^c w_1\wedge\cdots\wedge w_n$ and that
$[\chi_{\{u<v-\varepsilon+\varphi^j\}}]_{j=1}^{\infty}$ converges
to $\chi_{\{u<v-\varepsilon\}}$ a.e. w.r.t.
$\chi_{\{v>-\infty\}}\ddcn{v}$. The monotone
convergence theorem yields that
\begin{multline*}
\frac{1}{n!}\int_{\{u<v-\varepsilon\}}(v-\varepsilon-u)^n dd^c
w_1\wedge\cdots\wedge dd^cw_n
+\int_{\{u<v-\varepsilon\}}(-w_1)\ddcn{v}
 \leq \\ \leq
 \int_{\{u\leq v-\varepsilon\}}(-w_1)\ddcn{u}\, .
\end{multline*}
Inequality~(\ref{xing_ineq}) is now obtain by letting $\varepsilon\to 0^+$.
\end{proof}
\begin{corollary}\label{xing_cor_CP} Let $u,v,H\in\E$ be such that
 $\ddcn{u}$ vanishes on all pluripolar sets in $\Omega$ and $\ddcn{u}\leq\ddcn{v}$. Consider the following two conditions
 \begin{enumerate}

 \item $\displaystyle\varliminf_{z\to\zeta} (u(z)-v(z))\geq 0$ for every
$\zeta\in\partial\Omega$,

\item $u\in\Ker (H)$, $v\leq H$.

 \end{enumerate}
 If one of the above conditions is satisfied, then $u\geq v$ on
 $\Omega$.
\end{corollary}
\begin{proof} Assume that $u,v\in\E$ are such that
 $\ddcn{u}$ vanishes on all pluripolar sets in $\Omega$ and
 $\ddcn{u}\leq\ddcn{v}$.

\bigskip

(1): Moreover, assume that
\[
\varliminf_{z\to\zeta} (u(z)-v(z))\geq 0
\]
for every $\zeta\in\partial\Omega$. Let $\varepsilon > 0$.
Theorem~4.9 in~\cite{Khue_Hiep} implies that
\begin{multline}\label{xing_cor_CP_1}
\frac{\varepsilon^n}{n!} C_n(\{u+2\varepsilon<v\})\\ \leq
\sup\left\{\frac{1}{n!}\int_{\{u+2\varepsilon<v\}}(v-u-2\varepsilon)^n\ddcn{w}:
w\in\PSH{\Omega},\; 0\leq w\leq 1\right\}\\
\leq\sup\left\{\frac{1}{n!}\int_{\{u+\varepsilon<v\}}(v-u-\varepsilon)^n\ddcn{w}:
w\in\PSH{\Omega},\; 0\leq w\leq 1\right\}\\
\leq
\frac{1}{n!}\int_{\{u+\varepsilon<v\}}(-w)[\ddcn{u}-\ddcn{v}]\leq
0\, .
\end{multline}
Thus, $u+2\varepsilon\geq v$. Let $\varepsilon\to 0^+$, then
$u\geq v$ on $\Omega$.

\bigskip

(2): In this case assume that $u\in \Ker (H)$ and $v\leq H$. Since
$u\in\Ker (H)$ there exists a function $\varphi\in\Ker$ such that
$H+\varphi\leq u\leq H$. Let $\varphi^j$ be defined as in
Definition~\ref{maj_def_seq} and let $\varepsilon >0$. Similarly
as in~(\ref{xing_cor_CP_1}) we get that $u+2\varepsilon\geq
v+\varphi^j$. Let $\varepsilon\to 0^+$. Hence $u\geq v$ on
$\Omega$.
\end{proof}
\begin{remark} In Corollary~\ref{xing_cor_CP}, the assumption that $\ddcn{u}$ vanishes on
all pluripolar sets is essential.
\end{remark}
\begin{lemma}\label{DP_lem_ineq}
 Let $u,v\in\Ker(H)$, be such that $u\leq v$ and
 $\int_{\Omega} (-\varphi)dd^c u\wedge T<+\infty$, $\varphi\in\PSH{\Omega}$, $\varphi\leq
0$. Then the following inequality holds
\begin{equation}\label{DP_lem_ineq_1}
\int_{\Omega} (-\varphi)dd^c u\wedge T\geq\int_{\Omega}
(-\varphi)dd^c v\wedge T\, ,
\end{equation}
where $T=dd^c w_2\wedge\cdots\wedge dd^c w_n$,
$w_2,\ldots,w_n\in\E$.
\end{lemma}
\begin{proof} Let $[\Omega_s]$ be a fundamental
sequence in $\Omega$. By the assumption that $u\in\Ker (H)$ there
exists a function $\psi\in\Ker$ such that $H\geq u\geq\psi + H$.
For each $j\in\N$ consider the  function defined by $v_j=\max
(u,\psi^j+v)$, where $\psi^j$ is defined as in
Definition~\ref{maj_def_seq}. This construction imply that
$v_j\in\E$, $v_j=u$ on $\comp\Omega_j$, $u\leq v_j$,  and $[v_j]$ is
an increasing sequence that converges pointwise to $v$ q.e. on
$\Omega$, as $j\to +\infty$. Theorem~\ref{maj_thm_appr} implies
that there exists a decreasing sequence $[\varphi_k]$,
$\varphi_k\in\Eo\cap C(\bar\Omega)$, that converges
pointwise to $\varphi$, as $j\to +\infty$. We have by Stokes'
theorem that for each $s\geq j$ it holds that
\[
\int_{\Omega_s}(-\varphi_k)dd^c u\wedge T
-\int_{\Omega_s}(-\varphi_k)dd^c v_j\wedge
T=\int_{\Omega_s}(v_j-u)dd^c \varphi_k\wedge T\geq 0\, .
\]
By letting $s\to +\infty$ we get that
\begin{equation}\label{DP_lem_ineq_2}
\int_{\Omega} (-\varphi_k)dd^c u\wedge T\geq\int_{\Omega}
(-\varphi_k)dd^c v_j\wedge T\, .
\end{equation}
The function $\varphi_k$ is bounded and therefore it follows
from~\cite[remark on pg.~175]{cegrell_gdm} that $(-\varphi_k)dd^c
v_j\wedge T$ converges to $(-\varphi_k)dd^c v\wedge T$ in the
weak$^*$-topology, as $j\to +\infty$, which yields that
\begin{equation}\label{DP_lem_ineq_3}
\lim_{j\to +\infty}\int_{\Omega}(-\varphi_k)dd^c v_j\wedge T\geq
\int_{\Omega}(-\varphi_k)dd^c v\wedge T\, .
\end{equation}
Inequality~(\ref{DP_lem_ineq_2}) and~(\ref{DP_lem_ineq_3}) imply
that inequality~(\ref{DP_lem_ineq_1}) holds for $\varphi_k$ and
the monotone convergence theorem completes this proof, when we let
$k\to +\infty$.
\end{proof}
%
%
%

%
%
\begin{corollary}\label{DP_cor_conv}
Let $H\in\E$ and $\varphi\in\PSH{\Omega}$, $\varphi\leq 0$. If
$\,[u_{j}]$, $u_{j}\in\Ker (H)$, is a decreasing sequence that
converges pointwise on $\Omega$ to a function $u\in\Ker (H)$ as
$j$ tends to $+\infty$, then
\begin{equation}\label{DP_cor_conv_1}
\lim_{j\rightarrow +\infty}\int_{\Omega}(-\varphi)\ddcn{u_{j}}=
\int_{\Omega}(-\varphi)\ddcn{u}\, .
\end{equation}
\end{corollary}
\begin{proof} Let $\varphi\in\PSH{\Omega}$, $\varphi\leq 0$, and let $u_j,u\in\Ker (H)$ be such
that $u\leq u_j$. If $\int_{\Omega}(-\varphi)\ddcn{u}=+\infty$,
then~(\ref{DP_cor_conv_1}) follows immediately and therefore we
can assume that $\int_{\Omega}(-\varphi)\ddcn{u}<+\infty$.
Lemma~\ref{DP_lem_ineq} implies that
$[\int_{\Omega}(-\varphi)\,\ddcn{u_{j}}]$ is an increasing
sequence that is bounded above by $\int_{\Omega}(-\varphi)\,\ddcn{u}$. From
Corollary~5.2 in~\cite{cegrell_gdm} it follows that the
sequence $[(-\varphi)\ddcn{u_j}]$ converges to
$(-\varphi)\ddcn{u}$ in the weak$^*$-topology, as $j\to +\infty$,
and the desired limit of the total masses is valid.
\end{proof}
%
%
%

%
%
\begin{lemma}\label{ID_lem}  Let $H\in\E$ and let $u,v\in\Ker (H)$ be such
that $u\leq v$. Then for all $w_j\in\PSH{\Omega}\cap
L^{\infty}(\Omega)$, $-1\leq w_j\leq 0$, $j=1,2,...,n$, $\int_\Omega
(-w_1)\ddcn{u}<+\infty$, we have that the following inequality
holds
\begin{multline}\label{ID_thm_1}
\frac{1}{n!}\int_{\Omega}(v-u)^n dd^c w_1\wedge\cdots\wedge
dd^cw_n +\int_{\Omega}(-w_1)\ddcn{v}
 \leq \\ \leq
 \int_{\Omega}(-w_1)\ddcn{u}\, .
\end{multline}
\end{lemma}
\begin{proof} First we assume that $u,v\in\Eo (H)$. By definition
there exists a function $\varphi\in\Eo$ such that $H\geq u\geq
\varphi +H$. For each $\varepsilon>0$ small enough choose
$K\Subset\Omega$ such that $\varphi \geq - \varepsilon$ on
$\comp{K}$. Hence,
\[
u\geq \varphi + H\geq -\epsilon +H \geq -\epsilon
+v\;\;\text{on}\;\; \comp{K}\, ,
\]
and therefore it follows that $\max (u , v- \epsilon ) =u$ on
$\comp{K}$. By using Proposition~3.1
in~\cite{Khue_Hiep} we get that
\begin{multline*}
\frac{1}{n!}\int_\Omega (\max(u,v-\varepsilon)-u)^n dd^c
w_1\wedge\cdots\wedge dd^cw_n + \int_\Omega (-w_1) (dd^c
\max(u,v-\varepsilon))^n\\ \leq \int_\Omega (-w_1) (dd^c u)^n\, .
\end{multline*}
By letting $\varepsilon\to 0^+$ we obtain
inequality~(\ref{ID_thm_1}) in the case when $u,v\in\Eo (H)$. Using the case when $u,v\in\Eo (H)$ together with
Proposition~\ref{maj_prop_appr} and
Corollary~\ref{DP_cor_conv} we complete the proof.
\end{proof}

An immediate consequence of Lemma~\ref{ID_lem} is the following
identity principle. Theorem~\ref{ID_thm} play a technical
prominent role in Section~\ref{sec_DP}. In particular, this generalizes
for example Lemma~6.3 in~\cite{Rashkovskii}, Theorem~3.15 in~\cite{cegrell_gdm},
and the corresponding result in~\cite{Khue_Hiep}.

\begin{theorem}\label{ID_thm} Let $H\in\E$. If $u,v\in\Ker (H)$ is such that $u\leq
v$, $\ddcn{u}=\ddcn{v}$ and $\int_{\Omega}(-w)\ddcn{u}<+\infty$
for some $w\in\E$ which is not identically $0$, then $u=v$ on
$\Omega$.
\end{theorem}
\begin{theorem}\label{xing_thm_dir_N(H)}  Assume that $\mu$ is a
non-negative measure defined on $\Omega$ by $\mu=\ddcn{\varphi}$,
$\varphi\in\Ker$ with $\mu (A)=0$ for every pluripolar set
$A\subseteq\Omega$. Then for every $H\in\E$ such that
$\ddcn{H}\leq\mu$ there exists a uniquely determined function
$u\in\Ker (H)$ such that $\ddcn{u}=\mu$ on $\Omega$.
\end{theorem}
\begin{proof} The uniqueness part of this theorem follows by
the comparison principle in Corollary~\ref{xing_cor_CP}. We will
proceed with the existence part. Theorem~\ref{maj_thm_appr}
implies that there exists a decreasing sequence $[H_k]$,
$H_k\in\Eo\cap C(\bar\Omega)$, that converges pointwise
to $H$, as $j\to +\infty$. Let $[\Omega_j]$ be a fundamental sequence
 in $\Omega$. For each $j,k\in\N$ let $H_{k}^j$ be the
function defined as in Definition~\ref{maj_def_seq}, i.e.,
\[
H_{k}^j= \sup\big\{\varphi\in\PSH{\Omega}: \varphi\leq
H_{k}\;\;\mbox{on}\;\; \comp{\Omega_{j}}\big\}\, ,
\]
Then $H_{k}^j\in\Eo (\Omega)$ and $H_{k}^j$ is maximal on
$\Omega_j$. Consider the measure $\mu_j=\chi_{\Omega_j}\mu$
defined on $\Omega$, where $\chi_{\Omega_j}$ is the characteristic
function for the set $\Omega_j$ in $\Omega$ . For each $j\in\N$
the measure $\mu_j$ is a compactly supported Borel measure defined
on $\Omega$, $\mu_j$ vanishes on all pluripolar sets in $\Omega$
and $\mu_j (\Omega_j)<\mu_j (\Omega)<+\infty$. Therefore it
follows from Lemma~5.14 in~\cite{cegrell_gdm} that there exists a
uniquely determined function $\varphi_j\in\F (\Omega_j)$ such that
$\ddcn{\varphi_j}=\mu_j$ on $\Omega_j$. Moreover, from Theorem~4.1
in~\cite{cegrell_bdd} it follows that there exists functions
$u_{j,k}\in \F (\Omega_j,H_k^j)$ such that $\ddcn{u_{j,k}}=\mu_j$
on $\Omega_j$. Corollary~\ref{xing_cor_CP} implies that
\begin{equation}\label{xing_thm_dir_N(H)_1}
H_k^j\geq u_{j,k}\geq \varphi_j +H_k^j\;\;\;\;\text{ on }\;\;
\Omega_j\, ,
\end{equation}
since $\ddcn{u_{j,k}}\leq\ddcn{(\varphi_j+H_k^j)}$ and $H_k^j$ is
maximal on $\Omega_j$. The comparison principle (Corollary~\ref{xing_cor_CP})
yields that $[u_{j,k}]_{k=1}^{\infty}$ is a decreasing sequence. Let
$k\to +\infty$ and set $u_j=\lim_{k\to +\infty}u_{j,k}$,
then~(\ref{xing_thm_dir_N(H)_1}) gives us that $H^j\geq u_j\geq
\varphi_j +H^j$ on $\Omega_j$, i.e., $u_j\in \F
(\Omega_j,H^j)\subseteq\Ker (\Omega_j,H^j)$. From the assumption
that $\mu\geq\ddcn{H}$ we get that
$\ddcn{u_j}=\mu_j=\chi_{\Omega_j}\mu=\mu\geq\ddcn{H}$ on
$\Omega_j$ and therefore it follows from
Corollary~\ref{xing_cor_CP} that $u_j\leq H$ on $\Omega_j$. The
construction of $\mu_j$ and the fact that $[\Omega_j]$ is an
increasing sequence imply that $\ddcn{u_j}=\ddcn{u_{j+1}}$ on
$\Omega_j$. Hence $[u_j]$ is decreasing and
\begin{equation}
H\geq u_j\geq \varphi +H\;\;\;\;\text{ on }\;\; \Omega_j\, .
\end{equation}
Thus, the function $u=(\lim_{j\to +\infty}u_j)\in\Ker (\Omega,H)$ is such that
$\ddcn{u}=\mu$ on $\Omega$.
\end{proof}
\begin{remark}
Let $\mu$ be a non-negative measure defined on $\Omega$ such that
it vanishes on pluripolar subsets of $\Omega$ and that there
exists a function $\varphi\in\PSH{\Omega}$, $\varphi< 0$, such
that $\int_{\Omega}(-\varphi)\,d\mu<+\infty$. Then it follows
from~\cite{cegrell_bdd} that there exists a uniquely determined
function $\varphi\in\Ker$ such that $\ddcn{\varphi}=\mu$.
\end{remark}
\section{Monge-Amp\`{e}re measures carried on pluripolar sets}\label{sec_DP}

 Lemma~\ref{DP_lem_sing1} is due to Demailly (\cite{demailly_Plen}). Here we include his
 proof in our setting.

\begin{lemma}\label{DP_lem_sing1}
Let $u,u_k,v\in\E$, $k=1,\ldots, n-1$, with $u\geq v$ on $\Omega$ and set $T=dd^c u_1\wedge\cdots\wedge dd^c u_{n-1}$.
Then
\[
\chi_{\{u=-\infty\}} dd^c u\wedge T\leq \chi_{\{v=-\infty\}} dd^c v\wedge T\, .
\]
\end{lemma}
\begin{proof} Let $\varepsilon> 0$. Set $w_j=\max((1-\varepsilon)u-j,v)$. Then
$w_j=(1-\varepsilon)u -j$ on the open set $\{v<-\frac{j}{\varepsilon}\}$ and therefore
we have that
\[
dd^c w_j\wedge T = (1-\varepsilon) dd^c u\wedge T \text{ on } \left\{v<-\frac{j}{\varepsilon}\right\}\, .
\]
Hence $dd^c w_j\wedge T\geq (1-\varepsilon) \chi_{\{u=-\infty\}} dd^c u\wedge T$. Let
$j\to +\infty$, then
\[
dd^c v\wedge T\geq  (1-\varepsilon)\chi_{\{u=-\infty\}} dd^c u\wedge T \text{ on } \Omega\, .
\]
The proof is completed as $\varepsilon\to 0^+$.
\end{proof}
\begin{remark} For $j=1,\ldots,n$, let $u_j,v_j\in\E$, and $u_j\geq v_j$, then Lemma~\ref{DP_lem_sing1} implies that
\[
\int_A dd^c u_1\wedge\cdots\wedge dd^c u_n\leq \int_A dd^c v_1\wedge\cdots\wedge dd^c v_n\, ,
\]
for every pluripolar Borel set $A\subseteq\Omega$.
\end{remark}
\begin{remark} Let $u,v\in\E$ and assume that $\ddcn{v}$ vanishes on pluripolar sets. If
$u\geq v$, then it follows from Lemma~\ref{DP_lem_sing1} that  $\ddcn{u}$ vanishes on pluripolar sets.
\end{remark}
\begin{definition}\label{DP_def_aux1} Let $u\in\E$ and $0\leq \tau$ be a bounded
lower semicontinuous function. Then we define
\[
u_\tau=\sup\{\varphi\in\PSH{\Omega}: \varphi\leq \tau^{1/n}u\}\, .
\]
\end{definition}
\noindent Definition~\ref{DP_def_aux1} yields the following
elementary properties:
\begin{enumerate}
\item If $u,v\in\E$ with $u\leq v$, then $u_\tau\leq v_\tau$. \\

\item If $u\in\E$, then $0\geq u_\tau\geq
\|\tau\|^{1/n}_{L^\infty(\Omega)}u\in\E$. Hence, by~\cite{cegrell_gdm} we have that
$u_\tau\in\E$.\\

\item If $\tau_1,\tau_2$ are bounded
lower semicontinuous functions with $\tau_1\leq\tau_2$, then $u_{\tau_1}\geq u_{\tau_2}$.\\

\item If $u\in\E$, then $\supp\ddcn{u_\tau}\subseteq\supp\,\tau$
and if $\supp\, \tau$ is compact then
$u_\tau\in\F$.\\

\item If $[\tau_j]$, $0\leq \tau_j$ is an increasing sequence of
bounded lower semicontinuous functions that converges pointwise to
a bounded lower semicontinuous function $\tau$, as $j$ tends to
$+\infty$, then $[u_{\tau_j}]$ is
a decreasing sequence that converges pointwise to $u_{\tau}$, as $j$ tends to $+\infty$.\\
\end{enumerate}

\begin{lemma}\label{DP_lem_sing2}
Let $u\in\E$ and let $K$ be a compact pluripolar subset of $\Omega$.
Then
\[\ddcn{u_K}=\chi_K\ddcn{u}\, , \]
where $u_{\chi_O}$ is as in Definition~\ref{DP_def_aux1} and
\[
u_K=\left(\sup\{u_{\chi_O}:\; K\subset O\subset \Omega, \; O
\text{ is open}\}\right)^*\, .
\]
\end{lemma}
\begin{proof}
Choose a decreasing sequence $[O_j]$, $O_j\subseteq\Omega$, such
that $K=\bigcap_j O_j$. Then $[u_{\chi_{O_j}}]$ is an increases
sequence that converges to $u_K$ outside a pluripolar set,  as
$j\to +\infty$, and $\supp \ddcn{u_K}\subseteq\bigcap\bar{O}_j=
K$. We have that $u_{\chi_{O_j}}=u$ on $O_j$ hence $\ddcn{u_{\chi_{O_j}}}\geq \chi_K \ddcn{u}$, so
$\ddcn{u_K}\geq\chi_K\ddcn{u}$. On the other hand, $u_K\geq u$ and therefore
we have by Lemma~\ref{DP_lem_sing1} that
\[
\int_K\ddcn{u_K}\leq \int_K\ddcn{u} \text{ and }
\ddcn{u_K}=\chi_K\ddcn{u}\, .
\]
\end{proof}
\begin{lemma}\label{DP_lem_CS} Let $u_1,\ldots,u_n\in\E$. Then
\[
\int_A dd^c u_1\wedge\cdots\wedge dd^c u_n\leq
\left(\int_A\ddcn{u_1}\right)^{1/n}\cdots\left(\int_A\ddcn{u_n}\right)^{1/n}\,
,
\]
for every pluripolar Borel set $A\subset\Omega$.
\end{lemma}
\begin{proof} Without loss of generality we can assume that $A$ is
a compact pluripolar set and $u_1,\ldots,u_n\in\F$. Let $[G_j]$ be
a decreasing sequence of open subsets of $\Omega$ with $\bigcap_j
G_j =A$. Corollary~5.6 in~\cite{cegrell_gdm} yields that
\[
\int_\Omega dd^cu_{1_{G_j}}\wedge\cdots\wedge dd^cu_{n_{G_j}}\leq
\left(\int_\Omega
\ddcn{u_{1_{G_j}}}\right)^{1/n}\cdots\left(\int_\Omega
\ddcn{u_{n_{G_j}}}\right)^{1/n}
\]
For $1\leq k\leq n$ we have that $u_{k_{G_j}}=u_k$ on $G_j$ and
$\text{supp}(dd^cu_{k_{G_j}})^n\subset\bar {G}_j\subset\bar
{G}_1$, hence
\[
\int_{G_j} dd^cu_1\wedge\cdots\wedge dd^cu_n\leq\left(\int_{\bar
{G}_1} \ddcn{u_{1_{G_j}}}\right)^{1/n}\cdots\left(\int_{\bar
{G}_1}\ddcn{u_{n_{G_j}}}\right)^{1/n}\, .
\]
Let $j\to +\infty$. Lemma~\ref{DP_lem_sing2} then yields that
\[
\int_A dd^cu_1\wedge \cdots\wedge dd^cu_n\leq \left(\int_{\bar
{G}_1}\ddcn{u_{1_A}}\right)^{1/n}\cdots\left(\int_{\bar
{G}_1}\ddcn{u_{n_A}}\right)^{1/n}\, .
\]
\end{proof}

 For $u\in\E$ we write $\mu_u=\chi_{\{u=-\infty\}}\ddcn{u}$ and
define $S$ to be the class of simple functions
$f=\sum_{j=1}^m\alpha_j\chi_{E_j}$, $\alpha_j>0$, where $E_j$ are
pairwise disjoint and $\mu$-measurable such that $f$ is compactly
supported and vanishes outside $\{u=-\infty\}$. We write $T$ for
functions in $S$ where the $E_j$'s are compact.

\begin{definition}\label{DP_def_aux2}
 Let $u\in\E$ and $0\leq g\leq 1$ be a $\mu_u$-measurable function. We define
\[
u^g=\inf_{f\in T \atop f\leq g} \left(\sup\{u_\tau: f\leq \tau,\;
\tau \text{ is a bounded lower semicontinuous
function}\}\right)^*\, .
\]
\end{definition}
From Definition~\ref{DP_def_aux2} it follows that $u\leq u^g\leq
0$ and if $g_1\leq g_2$, then $u^{g_1}\geq u^{g_2}$. Furthermore, if $g\in T$,
then
\[
u^g=\left(\sup\{u_\tau: g\leq \tau,\;\tau \text{ is a bounded
lower semicontinuous function}\}\right)^*\in\F\, .
\]

\begin{lemma}\label{DP_lem_simple} Let $u\in\E$ and $g\in S$, then
$u^g\in\F$ and $\ddcn{u^g}=g\ddcn{u}$.
\end{lemma}
\begin{proof} Assume first that $g\in T$. Then $u^g\in\F$ as
already noted. Let $g=\sum_{k=1}^m \alpha_k \chi_{A_k}$ and
consider $u_k=u^{\alpha_k \chi_{A_k}}$. Then for $1\leq k\leq m$
we have that $u_1+\ldots + u_m\leq u^g\leq u_k$ so if
$B\subseteq\bigcup_{k=1}^{m} A_k$, then it follows from
Lemma~\ref{DP_lem_sing1} that
\[
\int_B\ddcn{u_k}\leq\int_B\ddcn{u^g}\leq\int_B\ddcn{(u_1+\ldots+u_m)}
\;\;\;\; 1\leq k\leq m\, .
\]
Hence, if $B\subset A_k$, then it follows from
Lemma~\ref{DP_lem_sing2} that $\int_B
\ddcn{u_k}=\alpha_k\int_B\ddcn{u}$ and from Lemma~\ref{DP_lem_CS}
we have that
\[
\alpha_k\int_B \ddcn{u}=\int_B\ddcn{(u_1+\ldots +u_m)}\, .
\]
Hence,
\[
\alpha_k\int_B\ddcn{u_k}\leq\int_B\ddcn{u^g}\leq
\alpha_k\int_B\ddcn{u}\;\;\;\; 1\leq k\leq m\, .
\]
for all Borel sets $B\subset A_k$, $k=1,...,m$. Thus
$\ddcn{u^g}=g\ddcn{u}$.

 Assume now that $g\in S$, i.e., $g=\sum_{j=1}^m\alpha_j\chi_{E_j}$,
 $\alpha_j>0$, $E_j$ are pairwise disjoint and $\mu$-measurable
 such that $g$ is compactly supported and vanishes
 outside $\{u=-\infty\}$. Choose to each $E_j$, $1\leq j\leq m$, increasing
 sequences $[K_j^p]_{p=1}^\infty$ of compact subsets of $E_j$ such
 that $\chi_p=\sum_{j=1}^m\chi_{K_j^p}$ converges to
 $\sum_{j=1}^m\chi_{E_j}$ a.e. w.r.t. $\mu$, as $p\to +\infty$.
 Then $\chi_p\in T$ and $g\chi_p\in T$. Furthermore, if $f_0\in T$
 with $f_0\leq g$, then $f_0\chi_p\in T$ and $f_0\chi_p\leq
 g\chi_p$. Hence $u^{f_0\chi_p}\geq u^{g\chi_p}$. By the first
 part of the proof we have that
 $\ddcn{u^{f_0\chi_p}}=f_0\chi_p\ddcn{u}$ and
 $\ddcn{u^{g\chi_p}}=g\chi_p\ddcn{u}$. Theorem~\ref{ID_thm}
 implies that $\lim_{p\to +\infty}u^{f_0\chi_p}=u^{f_0}$, hence
 $u^g\geq\lim_{p\to +\infty}u^{g\chi_p}$. Thus,
 $u^g=\lim_{p\to +\infty}u^{g\chi_p}\in\F$ and $\ddcn{u^g}=g\ddcn{u}$.
\end{proof}
\begin{theorem}\label{DP_thm_simple} Let $u\in\E$ and let $0\leq g\leq 1$
be a $\mu_u$-measurable function that vanishes outside
$\{u=-\infty\}$. Then $u^g\in\E$ and $\ddcn{u^g}=g\ddcn{u}$.
\end{theorem}
\begin{proof} Let $[g_j]$, $g_j\in S$, be an increasing sequence
that converges pointwise to $g$, as $j\to +\infty$. If $f\in T$ with $f\leq g$,
then by Lemma~\ref{DP_lem_simple} we have that $\min(f,g_j)\in S$
and $\ddcn{u^{\min(f,g_j)}}=\min(f,g_j)\ddcn{u}$. From
Theorem~\ref{ID_thm} it follows that $[u^{\min(f,g_j)}]$ is a
decreasing sequence that converges pointwise to $u^f$, as $j\to
+\infty$. Thus, $u^f\geq \lim_{j\to +\infty} u^{g_j}$ for every
$f\in T$ with $f\leq g$. Definition~\ref{DP_def_aux2} yields that
$u^g=\lim_{j\to +\infty} u^{g_j}$ and therefore it follows from
Lemma~\ref{DP_lem_simple} that $u^g\in\E$ and
$\ddcn{u^g}=g\ddcn{u}$.
\end{proof}
\begin{remark} Let $u$ and $g$ be as in Theorem~\ref{DP_thm_simple}.
If $\ddcn{u}$ vanishes on pluripolar sets, then it follows
from Theorem~\ref{DP_thm_simple} and the remark
after Lemma~\ref{DP_lem_sing1} that $u^g=0$ on $\Omega$.
\end{remark}
\begin{corollary} Let $u\in\E$ and $f,g$, $0\leq f,g\leq 1$,
be two $\mu_u$-measurable functions which vanishes outside
$\{u=-\infty\}$. If $f=g$ a.e. w.r.t. $\mu_u$, then $u^f=u^g$.
\end{corollary}
\begin{proof} Let $u\in\E$ and assume for now that $f,g\in S$. Then by Lemma~\ref{DP_lem_simple} we have that
$u^f,u^g\in\F$, $u^f\geq u^{\max(f,g)}$ and
\[
\ddcn{u^f}=f\ddcn{u}=\max(f,g)\ddcn{u}=\ddcn{u^{\max(f,g)}}\, .
\]
Hence, by Theorem~\ref{ID_thm} we have that $u^f= u^{\max(f,g)}$. Similarly
we get that $u^g= u^{\max(f,g)}$. Thus, $u^f=u^g$.

 For the general case let $[\Omega_j]$ be a fundamental sequence and let $f,g$,
$0\leq f,g\leq 1$, be two $\mu_u$-measurable functions that vanishes outside
$\{u=-\infty\}$. Our assumption that $f=g$, implies that
$\chi_{\Omega_j}f=\chi_{\Omega_j}g$ a.e. w.r.t. $\mu$
and by the first part of the proof we get that
$u^{\chi_{\Omega_j}f}=u^{\chi_{\Omega_j}g}$. The proof is then
completed by letting $j\to +\infty$.
\end{proof}

Example~\ref{DP_exem} shows that there exists a measure
$g\ddcn{u}$ carried by a pluripolar set that is not a discrete measure.

\begin{example}\label{DP_exem}
Let $\mu$ be a positive measure with no atoms and with support in
a compact polar subset the unit disc $\D$ (see e.g.~\cite{ransford}; p.82, and~\cite{carleson};
chapter IV, Theorem~1) Let $u$ be the subharmonic Green potential of $\mu$.
Consider $\nu =\mu\times\cdots\times\mu$ ($n$-times) and
$v(z_1,\ldots,z_n)= \max(u(z_1),...,u(z_n))$ on
$\D\times\cdots\times\D$ ($n$-times). Then $v\in\F$,
$\ddcn{v}=\nu$, $\nu$ has no atoms and it is supported by a pluripolar set.
\end{example}
\begin{lemma}\label{DP_lem_meas}
Assume that $\alpha,\beta_1,\beta_2$ are non-negative
measures defined on $\Omega$ which satisfies the following
conditions:
\begin{enumerate}

\item $\alpha$ vanishes on every pluripolar
subset of $\Omega$,

\item there exists a pluripolar sets $A\subset\Omega$ such that
$\beta_1(\Omega\backslash A)=0$.

\item for every $\rho\in \Eo\cap C(\bar\Omega)$ it holds that
\[
\int_\Omega (-\rho)\,\beta_1\leq\int_\Omega
(-\rho)\,(\alpha+\beta_2)<+\infty\, .
\]
\end{enumerate}
Then we have that
\[
\int_\Omega (-\rho)\,\beta_1\leq\int_\Omega (-\rho)\,\beta_2\, ,
\]
for every $\rho\in\Eo\cap C(\bar\Omega)$.
\end{lemma}
\begin{proof} Since $A$ is pluripolar and $\Omega$ is
bounded there exists a function $\varphi\in\PSH{\Omega}$,
$\varphi\leq 0$, such that $A\subseteq \{\varphi=-\infty\}$.
Take $\rho\in\Eo\cap C(\bar\Omega)$ and set
$\rho_j=\max\left(\rho,\frac{\varphi}{j}\right)$. Then we have
that $\int_\Omega (-\rho_j)\,\beta_1\leq\int_\Omega
(-\rho_j)\,(\alpha+\beta_2)<+\infty$ and by letting $j\to
+\infty$ we get that
\[
 \int_{\{\varphi=-\infty\}} (-\rho) \,\beta_1\leq\int_{\{\varphi=-\infty\}}
(-\rho)\,(\alpha+\beta_2)\, .
\]
But $\alpha$ vanishes on pluripolar sets and $\beta_1$ and
$\beta_2$ are carried by sets contained in $\{\varphi=-\infty\}$.
Thus,
\[
\int_\Omega (-\rho) \,\beta_1\leq\int_\Omega (-\varphi) \,\beta_2\,
,
\]
for every $\rho\in\Eo\cap C(\bar\Omega)$.
\end{proof}

Let $u\in\E$, then by Theorem~5.11 in~\cite{cegrell_gdm} there
exist functions $\phi_u\in\Eo$ and $f_u\in
L^{1}_{loc}\big(\ddcn{\phi_u}\big)$, $f_u\geq 0$ such that
$\ddcn{u}=f_u\,\ddcn{\phi_u}+\beta_u$. The non-negative measure
$\beta_u$ is such that there exists a pluripolar set
$A\subseteq\Omega$ such that $\beta_u (\Omega\backslash A)=0$. In
Lemma~\ref{DP_lem_pp} we will use the notation that
$\alpha_u=f_u\,\ddcn{\phi_u}$ and $\beta_u$ refereing to this
decomposition.

\begin{lemma}\label{DP_lem_pp} Let $u,v\in\E$.
If there exists a function $\varphi\in\E$ such that
$\ddcn{\varphi}$ vanishes on pluripolar sets and $|u-v|\leq
-\varphi$, then $\beta_u=\beta_v$.
\end{lemma}
\begin{proof} Let $\Omega^\prime\Subset\Omega$. It follows
from Lemma~\ref{DP_lem_sing1} that there is no loss of generality
to assume that $u,v,\varphi\in\F$, since it is sufficient to prove
that $\beta_u=\beta_v$ on $\Omega^\prime$. The assumption that
$|u-v|\leq -\varphi$ yields that $v+\varphi\leq u$ and therefore
it follows from Lemma~\ref{DP_lem_ineq} that
\begin{equation}\label{DP_lem_pp_1}
\int_\Omega(-\rho)\ddcn{u}\leq\int
(-\rho)\ddcn{(v+\varphi)}<+\infty\, ,
\end{equation}
where $\rho\in\Eo$. Since
$\sum_{j=1}^{n}\binom{n}{j}\ddck{\varphi}{j}\wedge\ddck{v}{n-j}\ll
C_n$  we have that $\beta_{v+\varphi}=\beta_v$ and
\[
\alpha_{v+\varphi}=\alpha_v+\sum_{j=1}^{n}\binom{n}{j}\ddck{\varphi}{j}\wedge\ddck{v}{n-j}\,
.
\]
Lemma~\ref{DP_lem_meas} and inequality~(\ref{DP_lem_pp_1}) yields
that
\[
\int_\Omega (-\rho)\, \beta_u\leq \int_\Omega (-\rho)\,\beta_v\, ,
\]
for every $\rho\in\Eo$. In a similar manner we get that
\[
\int_\Omega (-\rho)\, \beta_v\leq \int_\Omega (-\rho)\,\beta_u\, ,
\]
for every $\rho\in\Eo$. From Lemma~3.1 in~\cite{cegrell_gdm} it
now follows that $\beta_u=\beta_v$.
\end{proof}
\begin{lemma}\label{DP_lem_DP} Let $H\in\E\cap\MPSH{\Omega}$.

\begin{enumerate}

\item If $v\in\Ker$, $\ddcn{v}$ is carried by a pluripolar set, and $\int_\Omega(-\rho)\ddcn{v}<+\infty$
for all $\rho\in\Eo\cap C(\bar \Omega)$, then
\[
u=\sup\left\{\varphi\in\PSH{\Omega}:\; \varphi\leq\min(v,H)\right\}\in\Ker (H)\, ,
\]
and $\ddcn{u}=\ddcn{v}$.

\item Assume that $\psi\in \Ker$, $\ddcn{\psi}$ vanishes on pluripolar sets, $v\in\Ker(H)$,
$\ddcn{v}$ is carried by a pluripolar set, and $\int_\Omega(-\rho)(\ddcn{\psi} + \ddcn{v})<+\infty$
for all $\rho\in\Eo\cap C(\bar \Omega)$. If $u$ is the function
defined on $\Omega$ by
\[
u=\sup\left\{\varphi:\; \varphi\in\B \left(\ddcn{\psi},v)\right)\right\}\, ,
\]
where
\[
\B \left(\ddcn{\psi},v\right)=\left\{\varphi\in\E:\;
\ddcn{\psi}\leq\ddcn{\varphi} \text{ and }
\varphi\leq v\right\}\, ,
\]
then $u\in\Ker (H)$ and $\ddcn{u}=\ddcn{\psi}+\ddcn{v}$.

\end{enumerate}
\end{lemma}
\begin{proof}
 (1): Since $\min(v,H)$ is a negative and upper semicontinuous function we
 have that $u\in\PSH{\Omega}$ and $H\geq u \geq v+H$. Furthermore, $u\in\Ker (H)$,
 since $v\in \Ker (H)$. By Theorem~\ref{maj_thm_appr} we can choose a decreasing
 sequence $[v_j]$, $v_j\in\Eo\cap C(\bar \Omega)$, that converge pointwise to $v$ as
 $j\to +\infty$, and use Theorem~\ref{xing_thm_dir_N(H)} to solve
 $\ddcn{w_j}=\ddcn{v_j}$, $w_j\in\Ker (H)$, $j\in\N$. Consider
 \[
 u_j=\sup\left\{\varphi\in\PSH{\Omega}:\; \varphi\leq\min(v_j,H)\right\}\in\Eo (H)\, .
 \]
Then $u_j\geq w_j$, so by Lemma~\ref{DP_lem_ineq} $\int_\Omega(-\rho)\ddcn{u_j}\leq \int_D(-\rho)\ddcn{w_j}$.
Corollary~\ref{DP_cor_conv} now yields that
\[
\int_\Omega(-\rho)\ddcn{u}\leq \int_\Omega(-\rho)\ddcn{v}\;\;\text{ for all } \rho\in\Eo\cap C(\bar \Omega)\, ,
\]
and therefore it follows that $\ddcn{u}$ is carried by $\{u=-\infty\}$ and since
$v\geq u \geq v+H$ it follows from Lemma~\ref{DP_lem_pp} that $\ddcn{u}=\ddcn{v}$.
Thus, part (1) of this proof is completed.

(2): Using the classical Choquet's lemma (see e.g.~\cite{klimek}) and
Proposition~4.3 in~\cite{Khue_Hiep} we derive that $u\in\E$ and
$\ddcn{u}\geq\ddcn{\psi}$. Note that $u\in\PSH{\Omega}$, $u\leq 0$, as soon as
$v$ is only negative and upper semicontinuous and $\B \left(\ddcn{\psi},v\right)\neq\emptyset$.
Theorem~5.11 in~\cite{cegrell_gdm} gives that $\ddcn{u}=\alpha +\beta$, where $\alpha$ and
$\beta$ are positive measures defined on $\Omega$, such that $\alpha$ vanishes on
all pluripolar sets and $\beta$ is carried by a pluripolar set. The function
$(\psi+v)$ belongs to $\B \left(\ddcn{\psi},v\right)$ and
therefore we have that $v+\psi\leq u\leq v$. Hence
$u\in\Ker (H)$. By Lemma~\ref{DP_lem_pp} we have that $\beta=(dd^cv)^n$,
and we have already noted that $\alpha\geq\ddcn{\psi}$. Proposition~\ref{maj_prop_appr} implies
that there exists a decreasing sequence, $[v_{j}]$, $v_{j}\in\Eo (H)$, that
 converges pointwise to $v$, as $j\to +\infty$. Now,
\[
\int_\Omega(-\rho)\big(\ddcn{\psi} + \ddcn{v_j}\big)<+\infty\;\; \text{ for all } \rho\in\Eo\cap C(\bar \Omega)
\]
so by the last remark in section~3 and Theorem~\ref{xing_thm_dir_N(H)}, there exists a unique
function $w_j\in\Ker (H)$ with $\ddcn{w_j}=\ddcn{\psi} + \ddcn{v_j}$. It follows
from Corollary~\ref{xing_cor_CP} that $w_j\in \B \left(\ddcn{\psi},v_j\right)$, so if
we let
\[
u_j=\sup\left\{\varphi:\; \varphi\in\B \left(\ddcn{\psi},v_j\right)\right\}\, ,
\]
then $[u_{j}]$ decreases pointwise to $u$, as $j\to +\infty$. Furthermore,
Lemma~\ref{DP_lem_ineq} implies that
\[
\int_{\Omega}(-\rho)\ddcn{u_j}\leq \int_{\Omega}(-\rho)\ddcn{w_j}= \int_{\Omega}(-\rho)\big(\ddcn{\psi} + \ddcn{v_j}\big)\, .
\]
Let $j\to +\infty$, then Corollary~\ref{DP_cor_conv} yields that
\[
\int_\Omega (-\rho) \ddcn{u}\leq \int_\Omega
(-\rho)\big(\ddcn{\psi}+\beta\big)\, ,
\]
hence $\int_\Omega (-\rho ) (\alpha + \beta )\leq \int_\Omega (-\rho ) \big( \ddcn{\psi}+\beta \big)$. Since we know that $\alpha\geq\ddcn{\psi}$ it follows that
for all $\rho\in\Eo\cap C(\bar \Omega)$ we have that $\int_\Omega \rho\alpha=\int_\Omega\rho\ddcn{\psi}$, and therefore
is $\alpha=\ddcn{\psi}$. Thus, this proof is completed.
\end{proof}
\begin{theorem}\label{DP_thm_sub}
Assume that  $\mu$ is a non-negative measure.
\begin{enumerate}

\item There exist functions $\phi\in\Eo$, $f\in
L^{1}_{loc}\big(\ddcn{\phi}\big)$, $f\geq 0$, such that
\[
\mu=f\,\ddcn{\phi}+\nu\, ,
\]
where the non-negative measure $\nu$ is carried by a pluripolar subset of $\Omega$.

\item If there exists a function $w\in\E$ with $\mu\leq\ddcn{w}$, then there exist functions
$\psi,v\in\E$, $v,\psi\geq w$, such that
\begin{align*}
&\ddcn{\psi} =  f\ddcn{\phi}\\
&\ddcn{v}  = \nu\, ,
\end{align*}
where $\nu$ is carried by $\{v=-\infty\}$.

\item If there exists a function $w\in\E$ with $\mu\leq\ddcn{w}$, then to every  $H\in\E\cap\MPSH{\Omega}$
 there exists a function $u\in\E$, $w+H\leq u\leq H$, with $\ddcn{u}=\mu$. In particular, if $w\in\Ker$,
 then $u\in\Ker (H)$.
\end{enumerate}
\end{theorem}
\begin{proof}
(1): This is Theorem~5.11 in~\cite{cegrell_gdm}.

(2): Using the Radon-Nikodym theorem and the decomposition in part (1) (with the same notation)
we obtain that
\[
f \ddcn{\phi}= \tau\chi_{\{w>-\infty\}}\ddcn{w}\;\;\text{ and }\;\; \nu=\tau\chi_{\{w=-\infty\}}\ddcn{w}\, ,
\]
where $0\leq \tau\leq 1$ is a Borel function. For each $j\in\N$, let $\mu_{j}$ be the measure defined by
$\mu_{j}= \min (\varphi,j) \ddcn{\phi}$. Hence, $\mu_{j}\leq
\ddcn{(j^{\frac{1}{n}}\,\psi)}$ and therefore by Ko{\l}odziej's
theorem (see~\cite{kolo_range}) there exists a uniquely determined function
$\psi_{j}\in\Eo$ such that $\ddcn{\psi_j}=\mu_j$. The comparison
principle (Corollary~\ref{xing_cor_CP}) imply that $\psi_j\geq w$
and that $[\psi_j]$ is a decreasing sequence. The function
$\psi=\lim_{j\to +\infty}\psi_j$ is then in $\E$ and
$\ddcn{\psi}=f\,\ddcn{\phi}$. Theorem~\ref{DP_thm_simple} implies
that exists a functions $v\in\E$ such that $\ddcn{v}=\nu$ and $v\geq w$. Thus,
\[
\ddcn{\psi}=f\ddcn{\phi}\;\;\text{ and }\;\;\ddcn{v}=\nu\, .
\]

(3): Continuing with the same notations as in part (1) and (2), we choose an increasing sequence of simple functions $[g_j]$, $\supp \,g_j\Subset\Omega$, that converges to
$g=\chi_{\{w=-\infty\}}\tau$, as $j\to+\infty$. By
Theorem~\ref{DP_thm_simple} we have that $w^{g_j}\in\F$, $\ddcn{w^{g_j}}=g_j\ddcn{w}$ and $[w^{g_j}]$ is a decreasing
sequence that converges pointwise to $w^g$, as $j\to+\infty$.
Moreover $w^g\geq w$. Hence
$\ddcn{w^g}=\chi_{\{w=-\infty\}}\tau\ddcn{w}$. Set
\[
u_j=\sup\left\{\varphi\in\B \left(\ddcn{\psi_j},\min(w^{g_j},
H)\right)\right\}\, ,
\]
where
\[
\B \left(\ddcn{\psi_j},\min(w^{g_j}, H)\right)=\left\{\varphi\in\E :\;
\ddcn{\psi_j}\leq\ddcn{\varphi} \text{ and } \varphi\leq \min(w^{g_j},H)\right\}\, .
\]
This construction implies that $[u_j]$ is a decreasing sequence. The sequence $[u_j]$ converges to some
plurisubharmonic function $u$, as $j\to +\infty$, and by Lemma~\ref{DP_lem_DP} $u_j\in\Ker (H)$ with  $\ddcn{u_j}=\ddcn{\psi_j}+\ddcn{w^{g_j}}$. Furthermore, we have that
that $w+H\leq u_j\leq H$. We conclude the proof by letting $j\to +\infty$.

\end{proof}
\begin{remark}
Theorem~\ref{DP_thm_sub} generalize Theorem~4.4
in~\cite{czyz_cegrell}, Theorem~6.2 in~\cite{cegrell_gdm}, and
Corollary~1 in~\cite{xing_dec}.
\end{remark}
\begin{remark} Let $u_1,\ldots,u_n\in\E$. Then it follows from
Theorem~\ref{DP_thm_sub} that there exists a function $u\in\E$
such that $\ddcn{u}=dd^c u_1\wedge\cdots\wedge dd^cu_n$.
\end{remark}

\end{document}